\mag=\magstep1

\documentclass[10pt,a4paper,dviout]{amsart} 
\usepackage{amssymb,latexsym} 
\usepackage{color} 
\usepackage{graphicx}
\usepackage{comment}
\usepackage{datetime}
\usepackage{comment}
\usepackage{fancyhdr}

\input xy
\xyoption{all}

\usepackage{amsmath, amssymb, eucal, amscd, color, amsthm}

\def\R{\Bbb R}
\def\si{\sigma}

\def\cal{\mathcal}
\def\co{{\cal O}}
\def\Q{\Bbb Q}
\def\part{\partial}

\def\we{\wedge}
\def\e{\epsilon}
\def\dis{\displaystyle}

\def\P{\mathbb P}
\def\A{{\mathbb A}}
\def\C{\mathbb C}
\def\C0{$C_0$}
\def\s{{\square}}
\def\bl{\bigl[}
\def\br{\bigr]}

\def\sd{\operatorname{sd}}
\def\id{\operatorname{id}}
\def\codim{{{\rm codim}\,}}
\def\ov{\overline}
\def\Max{{\rm Max}}
\def\p1{\prec}
\def\wj{\widetilde{j}}
\def\wk{\widetilde{k}}
\def\sing{{\operatorname{sing}}}
\def\<{\langle}
\def\>{\rangle}

\def\sd{\operatorname{sd}}

\def\id{\operatorname{id}}
\newcommand{\sign}{\operatorname{sign}}
\newcommand{\Alt}{\operatorname{Alt}}
\def\res{\operatorname{Res}}

\newtheorem{theorem}{Theorem}[section] 
\newtheorem{proposition}[theorem]{Proposition} 
 
\newtheorem{definition}[theorem]{Definition} 
\newtheorem{conjecture}[theorem]{Conjecture} 
\newtheorem{remark}[theorem]{Remark} 
\newtheorem{lemma}[theorem]{Lemma} 
\newtheorem{lemma-definition}[theorem]{Lemma-Definition} 
\newtheorem{proposition-definition}[theorem]{Proposition-Definition}

\newtheorem{notation}[theorem]{Notation}

\newcommand{\CC}{{\mathbb{C}}}
\newcommand{\PP}{{\mathbb{P}}}
\newcommand{\QQ}{{\mathbb{Q}}}
\newcommand{\RR}{{\mathbb{R}}}
\newcommand{\ZZ}{{\mathbb{Z}}}

\newcommand\Spec{\operatorname{Spec}}

\newcommand\CH{\operatorname{CH}}

\newcommand{\Ker}{\operatorname{Ker}}

\newcommand\sgn{\operatorname{sgn}}

\newcommand{\cC}{{\mathcal C}}
\newcommand{\cD}{{\mathcal D}}

\newcommand{\cN}{{\mathcal N}}
\newcommand{\cP}{{\mathcal P}}

\newcommand{\bL}{ {\mathbf{Li}} }

\newcommand{\al}{\alpha}  \newcommand{\ga}{\gamma} 
 \newcommand{\eps}{\epsilon} \newcommand{\la}{\lambda}
\newcommand{\Ga}{\Gamma}

\newcommand{\BP}{{\mathbb{P}}}

\newcommand{\sq}{\square} 

\newcommand{\ot}{\otimes}

\newcommand{\vphi}{{\varphi}}

\newcommand{\injto}{\hookrightarrow}

\newcommand{\mapright}[1]{%
  \smash{\mathop{%
    \hbox to 1cm{\rightarrowfill}}\limits^{#1} } } 

\newcommand{\smapr}[1]{%
  \smash{\mathop{%
    \hbox to 0.5cm{\rightarrowfill}}\limits^{#1} } } 
\newcommand{\maprb}[1]{%
  \smash{\mathop{%
    \hbox to 1cm{\rightarrowfill}}\limits_{#1} } } 
\newcommand{\mapleft}[1]{%
  \smash{\mathop{%
    \hbox to 1cm{\leftarrowfill}}\limits^{#1} } }

\newcommand{\maplb}[1]{%
  \smash{\mathop{%
    \hbox to 1cm{\leftarrowfill}}\limits_{#1} } }

\def\alt{\operatorname{alt}}

\makeatletter
  
  \@addtoreset{equation}{section}
 \makeatother

\makeindex
\usepackage{version}	

\begin{document}
\fancyhf{}
\pagestyle{fancy}
\fancyhead[CO]{K  Kimura}
\fancyhead[CE]{A conjecture of Goncharov  and the co-Lie algebra of Bloch-Kriz mixed Tate motives}
\fancyfoot[C]{\thepage}

\title{On a relation of a conjecture of Goncharov to the co-Lie algebra of Bloch-Kriz mixed Tate
motives
} 
\author{Kenichiro Kimura }
\maketitle

\begin{abstract} Goncharov defined for each field $F$ and an integer $n$ greater than
1  a certain group $B_n(F)$  which is the quotient of the $\Q$-vector space on non-zero elements
of $F$ by a certain  subspace.  We consider the possibility of defining a linear map from $B_n(F)$
to the co-Lie algebra of  the category of mixed Tate motives defined by Bloch and
Kriz,  in terms of motivic polylogarithms.  We give results which support this possibility  
assuming part of the conjecture
by Beilinson and Soul\'{e}  on vanishing of $K$-groups of fields.  
\end{abstract}

\vskip 0.3cm



\setcounter{tocdepth}{3}


\setcounter{tocdepth}{1}
\footnote{keywords:  $K$-theory of fields,  mixed Tate motives,  higher Chow cycles. }

\section{Introduction}

Let $F$ be a field.  Gocharov defined in \cite{Gon1}  and  \cite{Gon2}  certain groups  
$B_n(F)$ for $n\geq 2$,  which are functorial in $F$. Let $\Q[F^\times]$
be the vector space  freely generated by non-zero elements of $F$.  He defines certain subspaces $R_n(F)$  of the vector space $\Q[F^\times]$  inductively
on $n$,  and $B_n(F)$ is defined to be the quotient $\displaystyle 
\Q[F^\times]/R_n(F)$. We will recall the definition of $R_n(F)$ later. 
 For an element $x\in F^\times$ we denote by $\{x\}_n\in B_n(F)$ the
element represented by $x$.  There exist  linear maps  $\delta_n$
 which are described as follows.  When
$n=2$,  $\delta_2:\,\,B_2(F)\to \wedge^2F_\Q^\times$ is a map
such that
\[
\delta_2(\{x\}_2)=\left\{\begin{array}{cl}
 -(1-x)\wedge x,& x\neq 1\\
 0,&x=1.
 \end{array}
 \right.
 \]
 For $n\geq 3$  we have

\hfil$\delta_n:\,\,B_n(F)\to  B_{n-1}(F)\otimes F_\Q^\times,\quad \{x\}_n\mapsto \{x\}_{n-1}
\otimes x.$\hfil

\noindent Goncharov states a conjecture part of   which is described as follows.

\begin{conjecture}[part of Conjecture 2.1 of  \cite{Gon1}]
\label{Gon}
Let $F$ be a field.  There is an isomorphism
\[\Ker (\delta_n)\simeq K^{(n)}_{2n-1}(F)_\Q\]
where $K^{(n)}_{2n-1}(F)_\Q$ is the $n$-th graded quotient of the $\ga$-filtration 
on the $K$-group $K_{2n-1}(F)_\Q.$
\end{conjecture}
The groups $B_n(F)$ was first defined by Zagier in (\cite{Za}  \S 8) for number fields.  Main conjecture in op. cit. 
is equivalent to the above conjecture for number fields.  
We explain the relation of  Conjecture \ref{Gon} to the theory of mixed Tate motives.
The triangulated category of  motives over $F$  which we denote by $\text{DM}(F)$,  
is defined independently by Hanamura \cite{Ha}, Levine \cite{Lev2} and Voevodsky \cite{Vo}.
We recall the following conjecture. 
\begin{conjecture}[Beilinson-Soul\'{e}]
\label{BS}
Let $F$ be a field.  Then we have
\[K^{(q)}_{2q-p}(F)\simeq {\rm CH}^q({\rm Spec}\,\,F, 2q-p)=0\]
for $q>0$ and $p\leq 0.$
\end{conjecture}
By results of Borel (\cite{Bo1}, \cite{Bo2}) and Soul\'{e} (\cite{So}) Conjecture \ref{BS} holds when $F$ is a number field.
Assuming Conjecture \ref{BS} Levine \cite{Lev0}  defines a $t$-structure on the subcategory
of $\text{DM}(F)$ generated by Tate objects $\Q(n)$ for $n\in \ZZ$,  whose core is by definition
the category of mixed Tate motives which we denote by $\text{MT}(F)$. The category
$\text{MT}(F)$  is a neutral Tannakian category over $\Q$. The co-Lie
algebra of $\text{MT}(F)$  which we denote by $\cC(F)$, has a grading.  
There  is a direct sum decomposition  

\hfil$\cC(F)=\bigoplus_n  \cC_n(F)$.\hfil

\noindent When $F$ is a number field,
by the construction of Deligne and Goncharov in \cite{DG}  one has motivic
polylogarithm $L(x)$ for each $x\in F^\times$,  which is an  ind-object of $\text{MT}(F)$.  An explicit
description of $L(x)$ is given by Dupont  and  Fres\'{a}n (arXiv preprint 2305.00789). 
  For a positive integer $n$,  The Hodge
realization of $W_{2n} L(x)$  has a period matrix which is described in e.g. (29)  of
\cite{Du}.  The object $W_{2n} L(x)$  defines an element ${\rm Li}^\cC_n(x)$
in $\cC_n(F)$.  As is explained in Section 3.2 of \cite{Du}  there exists a well defined
linear map

\hfil$ \iota_n:\,\,B_n(F)\to \cC_n(F),\quad \{x\}_n\mapsto {\rm Li}^\cC_n(x)$\hfil

\noindent which is compatible with the map $\delta_n$ and the co-product $\delta$
of $\cC(F)$. The map $\iota_n$   induces a map

\hfil$\psi_n:\quad \Ker \delta_n\to K_{2n-1}(F)_\Q$.\hfil

\noindent Surjectivity of $\psi_n$  implies Zagier's conjecture on special
values of the Dedekind zeta function of $F$.  However if we do not assume Conjecture
\ref{BS}  the category  $\text{MT}(F)$  is no longer available.

On the other hand  Bloch and Kriz defines in \cite{BK} a category of mixed Tate
motives over any field $F$  without depending on Conjecture \ref{BS}. They define
a certain differential graded algebra  $\cN$  from the complex of algebraic cycles
on affine spaces over $F$. The bar complex of $\cN$,  which we denote by $B(\cN)$,
has a product and a coproduct.  The cohomology of dimension zero  $H^0(B(\cN))$
is a commutative Hopf algebra with Adams grading.  We denote this Hopf algebra 
by $\chi_F$.  The category of mixed Tate motives due to Bloch and Kriz 
is defined to be that of finite dimensional graded co-modules over $\chi_F$. 
We denote this category  by $\text{MT}_{\text BK}(F)$. 
Let $\chi_F^+$ be the augmentation ideal of $\chi_F$.  The co-Lie algebra
of $\text{MT}_{\text BK}(F)$  is equal to $\chi_F^+/(\chi_F^+)^2$. In \cite{BK} a motivic
polylogarithm object $\bL_n(x)$ of $\text{MT}_{\text BK}(F)$
 is defined for each $x$ in $F^\times$ and for each $n\geq 2$.
Each object  $\bL_n(x)$ defines an  element of $\chi_F^+/(\chi_F^+)^2$.  We consider 
the following question:  {\it  Is there a well defined linear map from $B_n(F)$ 
to $\chi_F^+/(\chi_F^+)^2$ which sends $\{x\}_n$ to $\bL_n(x)$?} Note that we consider all fields,
not only number fields.   In this paper
we give a  result which partially answers this question  in positive direction.  
For each $n\geq 2$  we define a subspace of $R_n(F)$  inductively on $n$,
which we denote by $R_n'(F).$  The group $B_n'(F)$ is defined to be the quotient
$\Q[F^\times]/R'_n(F)$.  An element $x\in F^\times$ regarded as an element of $\Q[F^\times]$
is denoted by $\{x\}$. The element of $B_n'(F)$ represented by $x$ is denoted by
$\{x\}_n'$. 
 The linear map $\delta'_n:\,\Q[F^\times]
\to B_{n-1}'\otimes F_\Q^\times$  is defined to be the map which
sends  $\{x\}$ to $\{x\}'_{n-1}\otimes x$. Unlike the case of the groups $B_n(F)$,  we do not know
  whether the subspace $R_n'(F)$ is contained in $\Ker \delta'_n$. (see Remark \ref{remdelta}).
We explain  our results.  Let $X=\sum_i a_i\{x_i\}$ be an element of $\Q[F^\times].$
Suppose that $X$ is contained in $\Ker \delta_n'$. 
In the assertion (1) of Theorem \ref{existc}   together with the assertion   (1) of Theorem  
 \ref{poly}  we show that there exists a  cocycle $\cP_n(X)$  of $\cN$  which defines an element of 
$CH^n(\text{Spec} F, 2n-1)=K^{(n)}_{2n-1}(F)_\Q$,  such that the image of 
$\cP_n(X)$ in $\chi_F^+/(\chi_F^+)^2$  is equal to $\sum_ia_i\bL_n(x_i)$.  In this case the 
image of $X$ in $B_n(F)$ is contained in $\Ker \delta_n$, so our result  suggests that 
the element of $CH^n(\text{Spec} F, 2n-1)$ defined by $\cP_n(X)$  is equal to
$\psi_n(X)$.  The assertion (2) of Theorem \ref{existc} means that if $X$ in contained
in $R_n'(F)$,  then there exists  a cocycle $\cP_n(X)$ of $\cN$
whose image in  $\chi_F^+/(\chi_F^+)^2$ is equal to  $\sum_ia_i\bL_n(x_i)$,  and moreover
$\cP_n(X)$ is a coboundary. This implies  that the element
 $\sum_ia_i\bL_n(x_i)$ in $\chi_F^+/(\chi_F^+)^2$  is zero.

\noindent We mention \cite{GM}  and \cite{Zh}  as interesting related results.
Unfortunately  we could not completely remove the dependence on Conjecture \ref{BS}
from our argument.
We  need to assume the following conjecture,  which is the special case 
 of Conjecture \ref{BS}
where $p=0$. 
\begin{conjecture}
\label{BS0}
Let $F$ be a field.  Then we have
\[K^{(q)}_{2q}(F)\simeq {\rm CH}^q({\rm Spec}\,\,F, 2q)=0.\]
for $q>0$.
\end{conjecture}

This paper is organized as follows.  In Section 2 we recall the definition of Goncharov's groups
$B_n(F)$,  and we define the groups $B'_n(F)$.  Then we recall the definition of the algebra
$\cN$  from \cite{BK},  and state and prove our first main result Theorem
\ref{existc}. Crucial ingredients are certain cycles $\rho_k(x)$  in $\cN$  which are defined
in \cite{BK}.  They are also important for the definition of motivic polylogarithms 
$\bL_n(x)$  in  $\text{MT}_{\text BK}(F)$.  In Section 3 we recall the definition
of the bar complex $B(\cN)$,  the category $\text{MT}_{\text BK}(F)$  and motivic
polylogarithms $\bL_n(x)$.  Then we state and prove our second main result  
Theorem \ref{poly}.

\section{The groups $B_n'(F).$}
Let $F$ be a field, and let $n$ be an integer greater than 1. We recall the definition
of the groups $B_n(F)$  from \cite{Gon1}.  We regard $\Q[F^\times]$ as the quotient
 of $\Q[\P^1(F)]$ by the relation $\{\infty\}=\{0\}=0.$  Consider the map
 \[\widetilde{\delta}_2:\,\,\Q[F^\times]\to \wedge^2 F_\Q^\times,\,\,\{x\}\mapsto -(1-x)\wedge x,\,\,\{1\}\mapsto 0.\]
For $n\geq 3$,  suppose that the groups  $B_{n-1}(K)$  are defined for any field $K$,
and consider the map
\[\widetilde{\delta}_n:\,\,\Q[F^\times]\to B_{n-1}(K)\otimes K^\times_\Q,\,\,
\{x\}\mapsto \{x\}_{n-1}\otimes x.\]  
The subspace of $\Q[F^\times]$ generated by elements $\xi(1)-\xi(0)$  for
$\xi\in \Ker \widetilde{\delta}_n(F(t))$  is denoted by $R_n(F)$.  The group
$B_n(F)$ is by definition the quotient $\Q[F^\times]/R_n(F).$  By {\rm Lemma  1.16} of \cite{Gon2} 
the subspace $R_n(F)$ is contained in $\Ker \widetilde{\delta}_n$,  and so there is a well defined
map 

\hfil$\delta_n:\,\,B_n(F)\to  B_{n-1}(K)\otimes K^\times_\Q,\,\,
\{x\}_n\mapsto \{x\}_{n-1}\otimes x.$\hfil

\begin{remark}  The above definition of the group $B_n(F)$ is different from the one given by Zagier in \cite{Za}.
The group  $B_n(F)$  due to  Zagier  is equal to  the kernel of the map $\delta_n$  in Goncharov's $B_n(F)$.
\end{remark}

\noindent We give the definition of the groups $B'_n(F)$.  
 Let  $\delta_2'$ be the map equal to $\widetilde{\delta}_2$.  
For $n\geq 3$, suppose we have defined  the vector spaces $B'_{n-1}(K)$  for each field $K$.
Let $\delta'_n$ be the morphism  from $\Q[F^\times]\to B'_{n-1}(F)\otimes F_\Q^\times$  defined by 
\[\sum_ia_i\{x_i\}\mapsto \sum_i a_i\{x_i\}'_{n-1}\otimes x_i.\]
For an element $u$ of $F(t)^\times$ consider the following condition.
\[\{u(0), u(1)\}\cap \{0,\infty\}=\emptyset.\]
We denote this condition by \C0. 
The subspace $R'_n(F)$  of $\Q[F^\times]$ is defined to be the one generated by  $\al(1)-\al(0)$
where 
\[
\begin{aligned}
\al=&\sum_ja_j\{f_j(t)\}\in \Q[F(t)^\times],\, \al\in \Ker \delta_n'(F(t)) \text{\it and  for each }j\\ 
 &\text{\it the functions } \,f_j(t) \text{ and } 1-f_j(t) \, \text{\it  satisfy \C0}.
 \end{aligned}
 \]
 We define $B'_n(F)$ as the quotient $\Q[F^\times]/R'_n(F)$. 
\begin{remark} 
\label{remdelta}
It follows from the definition that $\Ker \delta_n'\subset \Ker \widetilde{\delta}_n$
since we can show inductively that $R'_n(F)$ is a  subspace of $R_n(F)$.
 Unlike the morphisms $\widetilde{\delta}_n$,   it is not clear  whether the subspace
$R'_n(F)$ is contained in $\Ker \delta'_n$.  Naive analogue of the proof for the case of
$B_n(F)$  i.e. {\rm Lemma  1.16} of \cite{Gon2}  does not seem to work.
\end{remark}

We recall the definition of the differential graded algebra $\cN$  from \cite{BK}.  
For a field $F$ set $\s^n=\s^n_F=(\bold P^1_F -\{1\})^n$, which is isomorphic to affine $n$-space 
as a variety, and set $(u_1, \cdots, u_n)$ to be the coordinates of $\s^n$.
For a quasi-projective variety $X$ over $F$, the faces of $X\times \s^n$ are subvarieties
defined by setting several coordinates $u_i$ to be zero or $\infty.$  We denote by 
$Z^p(X,n)$ the $\Q$-vector space generated by irreducible closed subsets of
$X\times \s^n$ of codimension $p$ which intersect all faces properly.  The linear map 

\hfil$\part_0^i:\quad Z^p(X,n)\to Z^p(X,n-1),\quad Z\mapsto Z\cdot(u_i=0)$\hfil

\noindent  resp.

\hfil$\part_\infty^i:\quad Z^p(X,n)\to Z^p(X,n-1),\quad Z\mapsto Z\cdot (u_i=\infty)$\hfil

is used to define the differential
\begin{equation}
\label{differential}
d=\sum_{i=1}^n(-1)^{i-1}(\part_0^i-\part_\infty^i):\quad  Z^p(X,n)\to Z^p(X,n-1).
\end{equation}
The group $G_n= \{\pm 1\}^n\rtimes S_n$\index{$G_n$} acts naturally on  $\s^n$ as follows. The subgroup
$\{\pm 1\}^n$ acts by the inversion of the coordinates $u_i$, and  the symmetric group $S_n$ acts by
permutation of $u_i$'s. This action induces an action of  $G_n$ on $Z^r(X, n)$.
Let $\sign: G_n \to \{\pm 1\}$ be the character which sends 
$(\eps_1, \cdots, \eps_n; \sigma)$ to $\eps_1\cdot \cdots\cdot\eps_n\cdot\sign(\sigma)$. 
We denote by $\Alt$ the idempotent 

\hfil$({1}/{|G_n|})\sum_{g\in G_n} \sign(g) g$     \index{$\Alt$}\hfil

\noindent of $\Q[G_n]$.  For a $\Q[G_n]$-module M, 
the submodule of $M$

\hfil$\{\al \in M\mid \Alt \al=\al\}=\Alt(M)$\hfil

\noindent is denoted by $M^{\alt}.$ By Lemma 4.3 \cite{BK}  the projector $\Alt$ and the differential
$d$ in (\ref{differential}) commute, and we can define a complex  $N^\bullet (X)(p)$
by setting $N^j (X)(p)=Z^p(X, 2p-j)^{\alt}$ and the differential

\hfil$d:\quad N^j (X)(p)\to N^{j+1}(X)(p)$\hfil

\noindent is the map $d$ defined in (\ref{differential}).  By Theorem 4.11 \cite{Lev}
we have the following.

\begin{proposition}
\label{alteq1}
Let $X$ be a smooth quasi-projective variety over $F$.  Then there exists an isomorphism
\[H^j(N^\bullet(X)(p))\simeq CH^p(X, 2p-j).\]
\end{proposition}
Let $s$ be a finite set of closed subvarieties of $X$.  We denote by $Z^p_s(X,n)$ the subspace
of $Z^p(X,n)$ generated by irreducible closed subsets of $X\times \s^n$ which intersect 
all the faces of $S\times\s^n$ properly  including $S\times\s^n$, for each $S\in s$.
$X$ is always contained in $s$. The subspace
$Z^p_s(X,n)^{\alt}$ of $Z^p(X,n)^{\alt}=N^p(X)(p)$ is denoted by $N^p_s(X)(p)$.
By Theorem 3.1 and the proof of
Theorem 4.11 of \cite{Lev} we have the following.
\begin{proposition}
\label{alteq2}
The inclusion  $N^p_s(X)(p)\to N^p(X)(p)$  is a quasi-isomorphism.
\end{proposition}
We recall the definition of a certain differential graded algebra (DGA) from \cite{BK}.
We denote by $N^\bullet(p)$ the complex $N^\bullet(\text{Spec } F)(p)$. There exists
a product  $N^i(p)\otimes N^j(q)\to N^{i+j}(p+q)$ defined by

\hfil$z\otimes w\mapsto z\cdot w:=\Alt(z\times w)=(-1)^{ij}w\cdot z.$\hfil

\noindent By this product the complex $\displaystyle \bigoplus_{p\geq 0} N^\bullet(p)$  is a
graded commutative DGA. Leibniz rule

\hfil$ d(z\cdot w)=dz\cdot w+(-1)^{\deg z}z\cdot dw$\hfil

\noindent is satisfied.  We denote this DGA by $\cN$.  We recall  the definition of the elements
of $\cN$ defined in \cite{BK} Section 9. For $k\geq 1$ and an element $a\in F^\times$
the element $\rho_k(a)\in N^1(k)$ is  defined as the locus

\hfil$(u_1,\cdots, u_{k-1},1-u_1,1-u_2/u_1,\cdots, 1-u_{k-1}/u_{k-2}, 1-a/u_{k-1})$\hfil

\noindent up to sign. (We omit the symbol ${}^{\alt}$ here and in the following.) The sign is defined as follows.
 We set 
\[
\rho_1(a)=\left\{ \begin{array}{cc}
(1-a)&  a\neq 1\\
0& a=1
\end{array}
\right.\]
and we define the signs of $\rho_k(a)\,\,(k\geq 2)$  so that  the equalities
\[d\rho_k(a)=\rho_{k-1}(a)\cdot (a)\]
hold.

\noindent For elements $a,\,b\in F^\times $, consider the cochain 
$u(a,b)$ in $N^0(1)=Z^1(F,2)^{\alt}$ which is parametrically described 
as the locus of $(x,\,\frac{x-a}{x-ab^{-1}}).$ This is  defined  in pp. 182 of \cite{To}.    One sees that 
\[du(a,b)=(b)+(ab^{-1})-(a)\]
and when $b=1$ the cochain $u(a,1)$ is zero. 
Let $\{z_j\}\,(j\in J)$ be a subset of $F^\times$ which forms  a basis of the vector space $F^\times_\Q$. For an element $x$ of $F^\times$  suppose that $(x)=\sum_j c_j (z_j)$ in 
$F^\times_\Q.$  Let $x[z]$ be the cocycle $\sum_j c_j(z_j)\in N^1(1).$  
We have an element $r(x,z)\in N^0(1)$ which is a linear
combination of  elements of the form $u(a,b)$, such that
$dr(x, z)=(x)-x[z].$  Let $\widetilde{P}_2(x)\in N^1(2)$ be the element 

\[\rho_2(x)+\rho_1(x)\cdot r(x,z)-r(\rho_1(x),z)\cdot x[z]\]
and for $n\geq 3$ let $P_n(x)\in N^1(n)$ be the element
\[\sum_{k=0}^{n-1} \frac1{k!}\rho_{n-k}(x)r(x,z)^k.\]
We have  equalities
\begin{equation}
\label{dp2}
d\widetilde{P}_2(x,z)=\rho_1(x)[z]\cdot x[z]
\end{equation}
and for $n\geq 3$
\begin{equation}
\label{dpn}
dP_n(x,z)=\left\{
\begin{aligned}
P_{n-1}(x)\cdot x[z] \quad&  n>3\\
\widetilde{P}_2(x,z)\cdot x[z]\quad & n=3.
\end{aligned}
\right.
\end{equation}

\begin{theorem}
\label{existc}  Let $F$ be a field. We assume {\rm  Conjecture \ref{BS0}}  holds for $F$.
\begin{itemize}
	\item[(1)] Let $X=\sum_ia_i \{x_i\}$ be an element of $\Q[F^\times].$ Suppose that $X$ is in 
	$\Ker \delta_n'$. Then for an integer $k$ such that $1\leq k\leq n-2$ and for each $k$-tuple 
	$(j_1,\cdots, j_k)$ of elements in $J$, there exists a cochain $D^{n-k}_{j_1\cdots j_k}\in N^0(n-k)$
	which has the following properties {\rm (I)}, {\rm (II)} and {\rm (III)}. 
		\begin{itemize}
		\item[(I)]  There are equalities
		\[dD^{n-k}_{j_1\cdots j_k}=\left\{
		\begin{aligned}
		&\sum_i a_ic_{ij_1\cdots j_k}P_{n-k}(x_i)-\sum_{j_{k+1}}D^{n-k-1}_{j_1\cdots j_{k+1}}\cdot (z_{j_{k+1}})
		&1\leq k\leq n-3 \\
		&\sum_i a_ic_{ij_1\cdots j_{n-2}}\widetilde{P}_2(x_i) & k=n-2
		\end{aligned}
		\right.
		\]
Here  we set  $(x_i)=\sum_{j}c_{ij}(z_j)$ in $F^\times_\Q$ 
for each $i$, and $c_{ij_1\cdots j_k}$ denotes $\displaystyle \prod_{t=1}^k c_{ij_t}.$ 

		\item[(II)] $D^{n-k}_{j_1\cdots j_k}=0$   if  there exists a $t,\,\,1\leq t\leq k$
		such that $c_{ij_t}=0$  for all $i$. If $k\geq 2$  the cochain  $D^{n-k}_{j_1\cdots j_k}$
		depends on the set $\{j_1,\cdots, j_k\}$,  but is independent of
		the ordering
		$j_1,\cdots, j_k.$
		\item[(III)]  The cochain 
		$$\cP_n(X):=\left\{\begin{aligned}
		&\sum_i a_i P_n(x_i)-\sum_{j_1} D^{n-1}_{j_1} \cdot (z_{j_1})
		&n\geq 3 \\
		&\sum_i a_i\widetilde{P}_2(x_i) & n=2
		\end{aligned}
		\right.
		$$ in $N^1(n)$ 
		is a cocycle.
		\end{itemize}
	\item[(2)]  If $X=\sum_ia_i \{x_i\}$ is in the subspace $R'_n(F)$ of $\Q[F^\times]$,  then there exist elements 
	$D^{n-k}_{j_1\cdots j_k}\in N^0(n-k)$  as in $(1)$  which have the properties $({\rm I}), 
	({\rm II}) \text{ and }
	 ({\rm III})$ as above,  and moreover the cochain $\cP_n(X)$ defined in $({\rm III})$ above is   a  coboundary.
\end{itemize}
\end{theorem}
For an element 
 $x\in F^\times$,  elements $r(x,z)\in N^0(1)$ such that $dr(x,z)=(x)-x[z]$ are not unique,  but if we assume 
Conjecture \ref{BS0} they are unique modulo adding  coboundary.  We have the following. 
\begin{proposition}
\label{modifyr}  
If the assertions of {\rm Theorem \ref{existc}} hold for a choice of $r(x_i,z)$, then  they hold for other choices
of $r(x_i,z)$  if we assume {\rm Conjecture \ref{BS0}.}

\end{proposition}
\begin{proof}
Choose one $i$ e.g. $i=1$ , and consider changing $r(x_1,z)$  to $r'(x_1,z):=r(x_1,z)+dE$  for a cochain 
$E\in N^{-1}(1)$.  We write $r(x_1,z)=r$  in this proof.  
 For $n-k\geq 3$ set 
\[P'_{n-k}(x_i):=\left\{
\begin{array}{lc}
\displaystyle \sum_{s=0}^{n-k-1} \frac1{s!}\rho_{n-k-s}(x_1)r'(x_1,z)^s  &i=1\\
P_{n-k}(x_i) &  i\neq 1
\end{array}
\right.\]
and 
\[\widetilde{P}'_2(x_i):=\left\{
\begin{array}{lc}
\displaystyle \rho_2(x_1)+\rho_1(x_1)\cdot r'(x_1,z)-r(\rho_1(x_1),z)\cdot x_1[z]  &i=1\\
\widetilde{P}_2(x_i,z) &  i\neq 1.
\end{array}
\right.\]
We have the equalities
\[\begin{aligned}
&P'_n(x_1)-P_n(x_1)\\
=& \sum_{s=1}^{n-1}\sum_{t=1}^s\frac1{s!}
\begin{pmatrix}
s\\
t
\end{pmatrix}
\rho_{n-s}(x_1)r^{s-t}\cdot(dE)^t
\end{aligned}
\]
and 
\[
\begin{aligned}
&d\left(\sum_{s=1}^{n-1}\sum_{t=1}^s\frac1{s!}
\begin{pmatrix}
s\\
t
\end{pmatrix}
\rho_{n-s}(x_1)r^{s-t}\cdot E\cdot(dE)^{t-1}\right)\\
=&\sum_{s=1}^{n-2}\sum_{t=1}^s\frac1{s!}
\begin{pmatrix}
s\\
t
\end{pmatrix}
\rho_{n-s-1}(x_1)\cdot (x_1)\cdot r^{s-t}\cdot E\cdot(dE)^{t-1}\\
-&\sum_{s=2}^{n-1}\sum_{t=1}^{s-1}\frac{(s-t)}{s!}
\begin{pmatrix}
s\\
t
\end{pmatrix}
\rho_{n-s}(x_1)r^{s-t-1}((x_1)-x_1[z])E\cdot (dE)^{t-1}\\
-&\sum_{s=1}^{n-1}\sum_{t=1}^s\frac1{s!}
\begin{pmatrix}
s\\
t
\end{pmatrix}
\rho_{n-s}(x_1)r^{s-t}\cdot(dE)^t\\
=&-\sum_{s=1}^{n-1}\sum_{t=1}^s\frac1{s!}
\begin{pmatrix}
s\\
t
\end{pmatrix}
\rho_{n-s}(x_1)r^{s-t}\cdot(dE)^t\\
&-\sum_{s=1}^{n-2}\sum_{t=1}^s\frac1{s!}
\begin{pmatrix}
s\\
t
\end{pmatrix}
\rho_{n-1-s}(x_1)r^{s-t}\cdot E \cdot (dE)^{t-1}\cdot x_1[z].
\end{aligned}
\]
For each $k$ such that $1\leq k\leq n-2$ and each $k$-tuple $j_1,\cdots, j_k$  set the cochain
$D^{n-k}_{j_1\cdots j_k}(\text{new})$ to be 

\[D^{n-k}_{j_1\cdots j_k}-a_1c_{1j_1\cdots j_k}
\sum_{s=1}^{n-k-1}\sum_{t=1}^s\frac1{s!}
\begin{pmatrix}
s\\
t
\end{pmatrix}
\rho_{n-k-s}(x_1)r^{s-t}\cdot E\cdot(dE)^{t-1}.\]
Then the elements $D^{n-k}_{j_1\cdots j_k}(\text{new})$ satisfy the equations 
in  (I) of the assertion (1)  for $P_{n-k}'(x_i)$.  The cochain
\[\cP_n(X)
+
d\left(-a_1\sum_{s=1}^{n-1}\sum_{t=1}^s\frac1{s!}
\begin{pmatrix}
s\\
t
\end{pmatrix}
\rho_{n-s}(x_1)r^{s-t}\cdot E\cdot(dE)^{t-1}\right)
\]
is equal to 
\[\sum_i a_i P'_n(x_i)-\sum_{j_1} D^{n-1}_{j_1}(\text{new}) \cdot (z_{j_1}).\]
We consider changing $r(\rho_1(x_1),z)$ to $r'(\rho_1(x_1),z):=r(\rho_1(x_1),z)+dE$
for a cochain $E\in N^{-1}(1).$  We need to change the elements $D^2_{j_1\cdots j_{n-2}}$
to

\hfil$D^2_{j_1\cdots j_{n-2}}(\text{new}):=D^2_{j_1\cdots j_{n-2}}-a_1c_{1j_1\cdots j_{n-2}}E\cdot x_1[z].$
\hfil

\noindent Then the cochain

\hfil$\sum_i a_ic_{ij_1\cdots j_{n-3}}P_3(x_i)-\sum_{j_{n-2}}D^2_{j_1\cdots j_{n-2}}(\text{new})\cdot (z_{j_{n-2}})
$\hfil

\noindent is the same as the old one since
\[\begin{aligned}
&-a_1c_{1j_1\cdots j_{n-3}}\sum_{j_{n-2}}c_{1j_{n-2}}E\cdot x_1[z]\cdot (z_{j_{n-2}})\\
=& -a_1c_{1j_1\cdots j_{n-3}}E\cdot x_1[z]\cdot x_1[z]\\
=&0.
\end{aligned}
\]
Hence we do not need to change cochains $D^{n-k}_{j_1\cdots j_k}$  for $k\geq 3$.
\end{proof}

\begin{proof}  The proof of Theorem \ref{existc}  is given by induction on $n$. First we prove the following lemma.

\begin{lemma}
\label{homotopy}
 We assume {\rm Conjecture \ref{BS0}}.  Let $k$ be an integer greater than $1$. 
We denote by  $\A^1_t$ the scheme ${\rm Spec} \,F[t]$.
\begin{itemize} 
\item [(1)] Let $U$ be an open subset
of $\A^1_t$ which is not empty,  and let $C$ be a cocycle in $N^1(U)(k)$. 
Suppose that the points 
$\{t=0\}$ and $\{t=1\}$ are on $U$.  We denote by 
$i_0$ resp. $i_1$ the closed immersion of the subvariety of $U\times \s^{2k-1}$
defined by  $t=0$ resp. $t=1$ to $U\times \s^{2k-1}.$
Suppose that $C$ is in $N_{\{0,\,1\}}^1(U)(k)$. 
The following holds.
\begin{itemize}
\item[(a)] The cocycle $i_0^*C-i_1^*C\in N^1(k)$ is a coboundary.

\item[(b)] If $i_0^*C$ is a coboundary, then $C$ is a coboundary.
\end{itemize}
\item[(2)]  Let $C_1$ be a cocycle of $N^1({\rm Spec}\,F(t))(k)$ and let $U_1$ be an open subset of $\A^1_t$ such that
the closure $\ov{C_1}$ of $C_1$ in $U_1\times \s^{2k-1}$ is a cocycle of $N^1(U_1)(k).$ If $C_1$ is a coboundary,
then $\ov{C_1}$ is also a coboundary.
\end{itemize}
\end{lemma}
\begin{proof}  We prove  (1).  We denote by $j_U$ the open immersion
of $U$ to $\A^1_t$. By \cite{B1} Corollary (0.2) we have an exact sequence
\[\cdots\to CH^k(\A^1_t, 2k-1) \overset{j_U^*}{\to}CH^k(U, 2k-1 ) \to CH^{k-1}(\A^1_t-U, 2k-2)\to\cdots.\]
By Conjecture \ref{BS0} we have $ CH^{k-1}(\A^1_t-U, 2k-2)=\{0\}.$ So there is a cocycle
$\widetilde{C}\in Z^k(\A^1_t,2k-1)$ such that $j_U^*\widetilde{C}=C$ in $CH^k(U, 2k-1 )$.
By Proposition \ref{alteq2} we can assume that the cocycle $\widetilde{C}$
belongs to $N^1_{\{0,1\}}(\A^1_t)(k)$. Then the assertion (1)  follows from
the homotopy invariance  $CH^k(\A^1_t, 2k-1)\simeq CH^k(F, 2k-1).$  For (2),  there exists a
cochain $D\in N^0(\text{Spec}\,F(t))(k)$ such that $dD=C_1$.  There exists a  non-empty open subset $U_2$ of $U_1$ 
such that the closure $\ov{D}$ of $D$ in $U_2\times \s^{2k-1}$  belongs to
 $N^0(U_2)(k)$,  
and the equality $d\ov{D}=\ov{C_1}$  holds in $N^1(U_2)(k)$.  There exists a cocycle $\widetilde{C_1}
\in Z^k(\A^1_t, 2k-1)$ such that $j_{U_1}^*\widetilde{C_1}=\ov{C_1}$ in $CH^k(U_1, 2k-1)$.  Let $p$ be a closed
point on $U_2$ such that $\ov{C_1}$ resp. $j_{U_2}^*\widetilde{C_1}$ resp.  $\ov{D}$ belongs to
$N^1_{\{p\}}(U_2)(k)$ resp. $N^1_{\{p\}}(U_2)(k)$ resp. $N^0_{\{p\}}(U_2)(k)$.
 Then we have $di_p^*\ov{D}=i_p^*\widetilde{C_1}$,
and so $\widetilde{C_1}$ is a coboundary by homotopy invariance.  Hence $\ov{C_1}$ is also a coboundary.
\end{proof}

We prove the assertion in the case where $n=2.$ In this case the assertions (I) and (II) of (1)  is empty.  We consider 
the assertion (III) of (1). By (\ref{dp2}) we have 
\[d\cP_2=\sum_i a_id\widetilde{P}_2(x_i,z)=\sum_ia_i\rho_1(x_i)[z]\cdot  x_i[z]=0.\]
We prove the assertion (2). We need to show that the cochain $\cP_2$ is a coboundary. 
 Let $F_0^\times$ be the subspace of $F(t)^\times_\Q$ generated by the elements of $F(t)^\times$
which satisfy \C0.   Let $\{y_l\}$ 
be a subset of $F(t)^\times$ such that $y_l$ satisfies \C0 for each $l$, and $\{y_l\}$ forms a  basis of $F_0^\times.$  Let $\{x_i\}$ be a finite subset of $F(t)^\times$,  and suppose that
for each $i$ the functions $x_i$ and $1-x_i$ satisfy \C0,  and  that 
$X:=\sum_i a_i\{x_i\}\in \Ker \delta'_2$ i.e. 

\hfil$\sum_ia_ix_i\wedge (1-x_i)=0$ in $(F(t)^\times_\Q)^{\wedge 2}.$\hfil

\noindent  Suppose that 
$(x_i)=\sum_l e_{il} (y_l)$ in $F^\times_0$  for each $i$,  and let
$x_i[y]\in N^1(\text{Spec}\,F(t))(1)$ be the cochain $\sum_l e_{il}(y_l)$.  
 There exists an open subset $U_t$ of $\A^1_t$ which contains the points $t=0$ and $t=1$,
 such that the closure of $\sum_ia_i\widetilde{P}_2(x_i)$
in $U_t\times \s^3$ is a cocycle in $N^1_{\{0,\,1\}}(U_t)(1)$ since intersecting properly with the faces is an open condition.
By Lemma \ref{homotopy} (1) (a)
the image of $(i_0^*-i_1^*)\cP_2(X)$ in  $\CH^2({\rm Spec}\,F, 3)$ is zero,  and there exists a cochain $D_2 \in N^0(2)$ such that

\hfil$dD_2=(i_0^*-i_1^*)\cP_2(X).$\hfil
 
\noindent Suppose that $(y_l(0))=\sum_j c_{lj} (z_j)$ in $F^\times_\Q$ for each $l$.

\noindent For each $j$ in $J$ let $z_j(s)$ be the linear
function such that $z_j(0)=1$ and $z_j(1)=z_j.$ Let 

$$m_l=\text{min}\{m\,|\, m>0, \,m\in \ZZ,\,\,\, mc_{lj}\in \ZZ
\text{ for all }j, \text{ and } y_l(0)^m=\prod_jz_j^{mc_{lj}} \text{ in } F^\times.\}$$
We have a cochain  $r(y_l(0),z)(s)$ in $N^0(\text{Spec}\,F(s))(1)$ which is a linear combination
of elements of the form $u(a,b)$, such that 
\[dr(y_l(0), z)(s)= (y_l(0))-\frac1{m_l}\left(\left(\frac{y_l(0)^{m_l}}{\prod_j z_j^{m_lc_{lj}}(s)}\right)+
\sum_jm_lc_{lj}(z_j(s))\right).\]
Let $Y_l(s)$ be the linear function such that $Y_l(0)=y_l(0)$ and $Y_l(1)=1.$ There is a
cochain $p(Y_l)(s)\in N^0(\text{Spec}\,F(s))(1)$ which is a linear combination of elements of the form
$u(a,b)$, such that $dp(Y_l)(s)=(Y_l(s))-\frac1{m_l}(Y_l(s)^{m_l})$ and $p(Y_l)(0)=r(y_l(0),z)(0).$
We set
\[r\bigl(x_i(0),y(0)\bigr)(s):= r(x_i(0), y(0))+\sum_l e_{il}\left(r(y_l(0),z)(s)-p(Y_l)(s)\right).\]
We have $r(x_i(0),y(0))(0)=r(x_i(0),y(0))$ and
\[dr(x_i(0), y(0))(1)=(x_i(0))-\sum_l e_{il}\left(\sum_j c_{lj}(z_j)\right).\]
Let $x_i(0)[y(0)](s)\in N^1({\rm Spec}\,F(s)(1)$ be the cochain
$$
\begin{aligned}
& (x_i(0))-dr\bigl(x_i(0),y(0)\bigr)(s)\\
=&\sum_je_{il}\bigl((y_l(0))-dr(y_l(0), z)(s)+dp(y_l)(s)\bigr)\\
=&\sum_le_{il}\left(\frac1{m_l}\left(\left(\frac{y_l(0)^{m_l}}{\prod_k z_k^{m_lc_{lj}}(s)}\right)+
\sum_jm_lc_{lj}(z_j(s))\right)+(Y_l(s))-\frac1{m_l}(Y_l(s)^{m_l})\right)
\end{aligned}
$$ and 
consider the cochain
\[
\begin{aligned}
&\widetilde{P}_2\bigl(x_i(0),y(0)\bigr)(s)\\
=&\rho_2\bigl(x_i(0)\bigr)+\rho_1\bigl(x_i(0)\bigr)\cdot 
r\bigl(x_i(0),y(0)\bigr)(s)-r\bigl(\rho_1(x_i(0)),y(0)\bigr)(s)
\cdot x_i(0)[y(0)](s)
\end{aligned}\]
in $N^1(\text{Spec}\,F(s))(2).$ Note that $\widetilde{P}_2(x_i(0),y(0))(0)=\widetilde{P}_2(x_i(0),y(0)).$  
The cochain  $\widetilde{P}_2\bigl(x_i(0),y(0)\bigr)(s)$ satisfies the equation
\begin{equation}
\label{dps2}
d\widetilde{P}_2(x_i(0),y(0))(s)=\rho_1(x_i(0))[y(0)](s)\cdot x_i(0)[y(0)](s).
\end{equation}
 One has a
similar cochain $\widetilde{P}_2(x_i(1),y(1))(s)$ in $N^1(\text{Spec}\,F(s))(2)$.  Consider the cochain
\[\cP_2(0)(s)=\sum_ia_i\widetilde{P}_2(x_i(0),y(0))(s) \text{ and } \cP_2(1)(s)=\sum_ia_i\widetilde{P}_2(x_i(1),y(1))(s).\]
in $N^1(\text{Spec}\,F(s))(2).$ We have

\hfil$\cP_2(0)(0)=i_0^*\cP_2(X)  \text{ and } \cP_2(1)(0)=i_1^*\cP_2(X).$\hfil

\noindent There exists an open set $U_s$ of $\A^1_s$ 
which contains $0$ and $1$, such that the closure
of  $\cP_2(0)(s)-\cP_2(1)(s)$ in $U_s\times \s^3$ is a cocycle in $N^1_{\{0,\,1\}}(U_s)(2)$. 
By the existence of the cochain $D_2$ and Lemma \ref{homotopy} (1) (b) the cochain  $\cP_2(0)(s)-\cP_2(1)(s)$
is  a coboundary,  and it follows that   $\cP_2(0)(1)-\cP_2(1)(1)$
is also a coboundary.  Hence we have the following conclusion.  For the element
$\sum_ia_i(\{x_i(0)\}-\{x_i(1)\})\in R'_2(F)$  there exists a cochain $D^2\in N^0(2)$ such that
\[dD^2=\sum_ia_i\bigl(\widetilde{P}_2(x_i(0))-\widetilde{P}_2(x_i(1))\bigr)\]
where
\[\widetilde{P}_2(x_i(0)):=\rho_2\bigl(x_i(0)\bigr)+\rho_1\bigl(x_i(0)\bigr)\cdot r\bigl(x_i(0),z\bigr)-
r\bigl(\rho_1(x_i(0)),z\bigr)\cdot x_i(0)[z]\]
and the cochain $\widetilde{P}_2(x_i(1))$ is defined similarly.

We assume that the assertions (1) and (2) hold for $n-1$  and consider the case of $n$.
Suppose that $\sum_ia_i\{x_i\}\in \Ker \delta'_n$  where
\[\delta'_n:\,\, \Q[F^\times]\to B'_{n-1}(F)\otimes F^\times_\Q,\,\,\{x\}\mapsto
\{x\}_{n-1}\otimes (x).\]
Since 
\[\delta'_n( \sum_ia_i\{x_i\})=\sum_{j_1} \bigl(\sum_i a_ic_{ij_1}\{x_i\}_{n-1}\bigr)\otimes z_{j_1}=0\]
we see that $\sum_i a_ic_{ij_1}\{x_i\}_{n-1}\in R'_{n-1}(F)$ for each $j_1.$
By induction hypothesis, 
for each $k$ such that $1\leq k\leq n-2$ and for each $k-$tuple $(j_1,\cdots, j_k)$ there exists a cochain
$D^{n-k}_{j_1j_2\cdots j_k}\in N^0(n-k)$ such that
\begin{equation}
\label{d1}
dD^{n-k}_{j_1j_2\cdots j_k}=\sum_ia_ic_{ij_1\cdots j_k}P_{n-k}(x_i)-\sum_{j_{k+1}}D^{n-k-1}_{j_1\cdots j_{k+1}}
\cdot (z_{j_{k+1}})
\end{equation}
for $1\leq k\leq n-3$, and
\begin{equation}
\label{d2}
dD^2_{j_1\cdots j_{n-2}}=\sum_i a_ic_{ij_1\cdots j_{n-2}}\widetilde{P}_2(x_i).
\end{equation}
Consider the cochain

\[\cP_n(X):=\sum_i a_iP_n(x_i)-\sum_{j_1}D^{n-1}_{j_1}\cdot (z_{j_1})\in N^1(n).\]
We want to modify this cochain to make it a cocycle.  The differentials of the
elements $D^2_{j_1\cdots j_{n-2}}$ depend on the sets $\{j_1,\,j_2,\cdots, j_{n-2}\}$, but are independent
of the ordering $j_1,\cdots, j_{n-2}$. So by  Conjecture \ref{BS0} there exist elements
$E^2_{j_1j_2\cdots j_{n-2}}\in N^{-1}(2)$ such that the elements
$dE^2_{j_1j_2\cdots j_{n-2}}+D^2_{j_1\cdots j_{n-2}}$ are independent of the ordering
$j_1,\cdots, j_{n-2}$. If $n-1\geq 4$  we have the equality 

\hfil$ D^2_{j_1\cdots j_{n-4}j_{n-3}j_{n-2}}=D^2_{j_1\cdots j_{n-4}j_{n-2}j_{n-3}}$\hfil

\noindent since
\[d^2D^4_{j_1\cdots j_{n-4}}=\left(\sum_{j_{n-3}<j_{n-2}} D^2_{j_1\cdots j_{n-4}j_{n-3}
j_{n-2}}-D^2_{j_1\cdots j_{n-4}j_{n-2}j_{n-3}}\right)(z_{j_{n-2}})\cdot (z_{j_{n-3}})=0.\]
So we can assume that $E^2_{j_1\cdots j_{n-4}j_{n-3}
j_{n-2}}=E^2_{j_1\cdots j_{n-4}j_{n-2}j_{n-3}}$.  Let $D^3_{j_1\cdots j_{n-3}}(\text{new})$
be the cochain 

\hfil$\displaystyle D^3_{j_1\cdots j_{n-3}}-\sum_{j_{n-2}}E^2_{j_1\cdots j_{n-4}j_{n-3}
j_{n-2}}\cdot (z_{j_{n-2}}).$\hfil

\noindent Then the cochain

\hfil$\dis \sum_i a_ic_{ij_1\cdots j_{n-4}}P_4(x_i)-\sum_{j_{n-3}}D^3_{j_1\cdots j_{n-4}j_{n-3}}(\text{new})\cdot 
(z_{j_{n-3}})$\hfil

\noindent is the same as the old one since

\[\begin{aligned}
 &\sum_{j_{n-3}}\left(\sum_{j_{n-2}}E^2_{j_1\cdots j_{n-4}j_{n-3}
j_{n-2}}\cdot (z_{j_{n-2}})\right)(z_{j_{n-3}})\\
=&\sum_{j_{n-3}<j_{n-2}}(E^2_{j_1\cdots j_{n-4}j_{n-3}
j_{n-2}}-E^2_{j_1\cdots j_{n-4}j_{n-2}j_{n-3}})(z_{j_{n-2}})\cdot (z_{j_{n-3}})\\
=&0.
\end{aligned}
\]
Hence we can replace $D^3_{j_1\cdots j_{n-4}j_{n-3}}$  with $D^3_{j_1\cdots j_{n-4}j_{n-3}}(\text{new}).$
Then the differentials of elements $D^3_{j_1\cdots j_{n-4}j_{n-3}}$ are independent of
the ordering  $j_1,\cdots, j_{n-3}.$  By Conjecture \ref{BS0} there exist elements
$E^3_{j_1\cdots j_{n-4}j_{n-3}}\in N^{-1}(3)$ such that the elements
$dE^3_{j_1\cdots j_{n-4}j_{n-3}}+D^3_{j_1\cdots j_{n-4}j_{n-3}}$ are independent of the
ordering $j_1,\cdots, j_{n-3}.$  We can continue this process until we obtain a set of
elements $D^{n-k}_{j_1j_2\cdots j_k}\in N^0(n-k)$
 for each $k-$tuple $(j_1,\cdots, j_k)$ which satisfy the equations (\ref{d1}) and (\ref{d2}),
 and which are independent of the ordering  $j_1,\cdots, j_k.$ For this new
 elements $D^{n-k}_{j_1j_2\cdots j_k}\in N^0(n-k)$ 
the cochain $\cP_n(X)$ is a cocycle. The assertion (1) is proved.

\noindent We consider the assertion (2). 
 Let $F_0^\times$ be the subspace of $F(t)^\times_\Q$ generated by the elements of $F(t)^\times$
which satisfy \C0.  Let $\{y_l\}$ 
be a subset of $F(t)^\times$ such that $y_l$ satisfies \C0 for each $l$, and $\{y_l\}$ forms a  basis of $F_0^\times.$  Let $\{x_i\}$ be a finite subset of $F(t)^\times$,  and suppose that
for each $i$ the functions $x_i$ and $1-x_i$ satisfy \C0,  and  that $X:=\sum_i a_i\{x_i\}\in 
\Ker \delta'_n(F(t))$. 
We suppose that for each $i$, we have $(x_i)=\sum_l e_{il}(y_l)$ in $F_0^\times.$ 
By the assertion (1) proved above, for each $k,\,\,1\leq k\leq n-2$ and for each $k$-tuple
$(l_1,\cdots, l_k)$,  there exists a cochain
$D^{n-k}_{l_1l_2\cdots l_k}\in N^0(\text{Spec}\,F(t))(n-k)$  which satisfies the equalities
\[dD^{n-k}_{l_1\cdots l_k}=\left\{
		\begin{aligned}
		&\sum_i a_ie_{il_1\cdots l_k}P_{n-k}(x_i)-\sum_{l_{k+1}}D^{n-k-1}_{l_1\cdots l_{k+1}}\cdot (y_{l_{k+1}})
		&1\leq k\leq n-3 \\
		&\sum_i a_ie_{il_1\cdots l_{n-2}}\widetilde{P}_2(x_i) & k=n-2
		\end{aligned}
		\right.
		\]
		and the cochain 
		$$\cP_n(X):=
		\sum_i a_i P_n(x_i)-\sum_{l_1} D^{n-1}_{l_1} \cdot (y_{l_1})
		$$ in $N^1(\text{Spec}\,F(t))(n)$ 
		is a cocycle. We can also assume that each cochain $D^{n-k}_{l_1\cdots l_k}$ is independent
		of the ordering $l_1,\cdots, l_k$.
For each $(n-2)$-tuple $l_1,\cdots, l_{n-2}$ there exists an open subset $U_t$ of $\A^1_t$ 
which contains the points $t=0$ and $t=1$, such that
the closure of $\sum_i a_ie_{il_1\cdots l_{n-2}}\widetilde{P}_2(x_i)$ in $U_t\times
\s^3$ is a cocycle of $N^1_{\{0,\,1\}}(U_t)(2)$. By Lemma \ref{homotopy} (2) 
and Proposition \ref{alteq2} there exists
a cochain $D^2_{l_1\cdots l_{n-2}}({\rm new})\in N^0_{\{0,\,1\}}(U_t)(2)$  such that

\hfil$ dD^2_{l_1\cdots l_{n-2}}({\rm new})=\sum_i a_ie_{il_1\cdots l_{n-2}}\widetilde{P}_2(x_i)$\hfil

\noindent  in $N^1_{\{0,\,1\}}(U_t)(2)$. By Conjecture \ref{BS0}
 there exists a cochain $E^2_{l_1\cdots l_{n-2}}\in
N^{-1}(\text{Spec}\, F(t))(2)$ such that

\hfil$dE^2_{l_1\cdots l_{n-2}}=D^2_{l_1\cdots l_{n-2}}({\rm new})-D^2_{l_1\cdots l_{n-2}}$\hfil

\noindent in $N^0(\text{Spec}\, F(t))(2)$. By the same argument as in the proof of (1) above  one can
inductively replace the elements $D^{n-k}_{l_1\cdots l_k}$  with  elements
$D^{n-k}_{l_1\cdots l_k}({\rm new})$  which are in $N^0_{\{0,\,1\}}(U_t)(n-k)$.  
Then by Lemma \ref{homotopy} (1) (a) the cochain $(i_0^*-i_1^*)\cP_n(X)
\in N^1(n)$  is a coboundary. 
Consider the cochain  

$$\sum_ia_ie_{il_1\cdots l_{n-2}}\widetilde{P}_2\bigl(x_i(0),y(0)\bigr)(s)
\in N^1(\text{Spec}\,F(s))(2)$$
for each $(n-2)$-tuple $l_1,\cdots,l_{n-2}$  defined in the same manner as 
in the case where $n=2$  of  the proof of the assertion (2).  There is an open subset
$U_s$ of $\A^1_s$ which contains the points $\{s=0\}$ and $\{s=1\}$, such that the closure of $\sum_ia_i e_{il_1\cdots l_{n-2}}\widetilde{P}_2\bigl(x_i(0),y(0)\bigr)(s)$ in
$U_s\times \s^3$ is a cocycle in $N^1_{\{0,\,1\}}(U_s)(2)$. Since $d(i_0^*D^2_{l_1\cdots l_{n-2}})=
\sum_ia_ie_{il_1\cdots l_{n-2}}\widetilde{P}_2\bigl(x_i(0),y(0)\bigr)(0),$\hfil

\noindent  by Lemma \ref{homotopy} (1) (b)  the closure of $\sum_ia_i
e_{il_1\cdots l_{n-2}}\widetilde{P}_2\bigl(x_i(0),y(0)\bigr)(s)$ in $U_s\times \s^3$ is a  coboundary.  Suppose that 

\hfil $dD_{l_1\cdots l_{n-2}}^2(0,s)=\sum_ia_ie_{il_1\cdots l_{n-2}}\widetilde{P}_2\bigl(x_i(0),y(0)\bigr)(s)$\hfil

\noindent  for a cochain $D^2_{l_1\cdots l_{n-2}}(0,s)\in N^0_{\{0,\,1\}}(U_s)(2)$.  Since the elements 
$D_{l_1\cdots l_{n-2}}^2(0, 0)$ and $i_0^*D^2_{l_1\cdots l_{n-2}}$
have the same coboundary $\sum_ia_ic_{il_1\cdots l_{n-2}}\widetilde{P}_2\bigl(x_i(0),z(0)\bigr)$, we can assume that $D_{l_1\cdots l_{n-2}}^2(0, 0)=i_0^*D^2_{l_1\cdots l_{n-2}}$  by adding a constant cocycle to $D^2_{l_1\cdots l_{n-2}}(0,s)$
if necessary.  For $m\geq 3$ consider the cochain  
\[P_m(x_i(0), y(0))(s):=\sum_{k=0}^{m-1} \frac1{k!}\rho_{m-k}(x_i(0))\{r(x_i(0),y(0))(s)\}^k\]
in $N^1(\text{Spec}\,F(s))(m)$. We have
\begin{equation}
\label{dpsn}
dP_m(x_i(0),y(0))(s)=\left\{
\begin{aligned}
P_{m-1}(x_i(0), y(0))(s)\cdot x_i(0)[y(0)](s) \quad&  m>3\\
\widetilde{P}_2(x_i(0),y(0))(s)\cdot x_i(0)[y(0)](s)\quad & m=3.
\end{aligned}
\right.
\end{equation}
For a $(n-3)$-tuple $l_1,\cdots,l_{n-3}$ consider the cochain
\[\begin{aligned}
&{\cal P}_3(0)_{l_1\cdots l_{n-3}}(s)\\
=&\sum_i a_ie_{il_1\cdots l_{n-3}}P_3(x_i(0),y(0))(s)-\sum_{l_{n-2}}
D^2_{l_1\cdots l_{n-2}}(0,s)\cdot [y_{l_{n-2}}(0)(s)]\\
&\in N^1(\text{Spec}\,F(s))(3)
\end{aligned}\]
where $[y_l(0)(s)]$ is the cochain 

\hfil$\displaystyle\frac1{m_l}\left(\left(\frac{y_l(0)^{m_l}}{\prod_j z_j^{m_lc_{lj}}(s)}\right)+
\sum_jm_lc_{lj}(z_j(s))\right)+(Y_l(s))-\frac1{m_l}(Y_l(s)^{m_l})$\hfil

\noindent  in $N^1(\text{Spec}\,F(s))(1).$ There exists an open subset $U_s$ of $\A^1_s$ which contains the points 
$0$ and $1$  such that the closure of ${\cal P}_3(0)_{l_1\cdots l_{n-3}}(s)$ in $U_s\times \s^5$ is a cocycle
in $N^1_{\{0,\,1\}}(U_s)(3)$.
By a similar argument as above  one can inductively find elements
$D_{l_1\cdots l_k}^{n-k}(0,s)$ and $D_{l_1\cdots l_k}^{n-k}(1,s)$ in $N^0_{\{0,\,1\}}(U_s)(n-k)$
for each $k,\,1\leq k\leq n-2$ and each $k$-tuple $\{l_1,\cdots, l_k\}$ which satisfy the equations
\[\begin{aligned}
&dD^{n-k}_{l_1\cdots l_k}(0,s)\\
=&\left\{
		\begin{aligned}
		&\sum_i a_ie_{il_1\cdots l_k}P_{n-k}(x_i(0),z(0))(s)-\sum_{l_{k+1}}D^{n-k-1}_{l_1\cdots l_{k+1}}
		(0,s)\cdot[y_{l_{k+1}}(0)(s)]
		&1\leq k\leq n-3 \\
		&\sum_i a_ie_{il_1\cdots l_{n-2}}\widetilde{P}_2(x_i(0),y(0))(s) & k=n-2
		\end{aligned}
		\right.
		\end{aligned}
		\]
and
\[\begin{aligned}
&dD^{n-k}_{l_1\cdots l_k}(1,s)\\
=&\left\{
		\begin{aligned}
		&\sum_i a_ie_{il_1\cdots l_k}P_{n-k}(x_i(1),y(1))(s)-\sum_{l_{k+1}}D^{n-k-1}_{l_1\cdots l_{k+1}}
		(1,s)\cdot [y_{l_{k+1}}(1)(s)]
		&1\leq k\leq n-3 \\
		&\sum_i a_ie_{il_1\cdots l_{n-2}}\widetilde{P}_2(x_i(1),y(1))(s) & k=n-2
		\end{aligned}
		\right.
		\end{aligned}
		\]
Let $\cP_n(X)(0)(s)$ resp.  $\cP_n(X)(1)(s)$  be the cochain 
\[\sum_ia_iP_n(x_i(0),y(0))(s)-\sum_{l_1}D_{l_1}^{n-1}(0,s)\cdot [y_{l_1}(0)(s)]\]
resp.
\[\sum_ia_iP_n(x_i(1),y(1))(s)-\sum_{l_1}D_{l_1}^{n-1}(1,s)\cdot [y_{l_1}(1)(s)]\]
in $N^1_{\{0,\,1\}}(U_s)(n)$.  The cochain $\cP_n(X)(0)(s)-\cP_n(X)(1)(s)$  is a cocycle, and
$\cP_n(X)(0)(0)-\cP_n(X)(1)(0)=(i_0^*-i_1^*)\cP_n(X)$ is a coboundary.   By Lemma \ref{homotopy} (1)
 (b) the cochain 
  $\cP_n(X)(0)(s)-\cP(1)(s)$ is a coboundary, and $\cP_n(X)(0)(1)-\cP_n(X)(1)(1)$  is also a coboundary.
The cochain $\cP_n(X)(0)(1)-\cP_n(X)(1)(1)$  satisfies (I) of the assertion (1).  We will show  that it can be
modified to satisfy the assertion  (II) of (1) too.  We are given  cochains $D^{n-k}_{j_1\cdots j_k}$
for each $k$-tuple $j_1,\cdots, j_k$ of elements of $J$  which satisfy the equations in
(I) of the assertion (1),  and moreover the cochain $\cP_n(X)$  in (III) of the assertion (1)
is a coboundary.  We call the process of modifying the cochains $D^{n-k}_{j_1\cdots j_k}$
to make them independent of the ordering $j_1,\cdots, j_k$ which is explained in the proof of the assertion
(1) {\it  symmetrization}.  When $\cP_n(X)$ is a cocycle, it is unchanged by the symmetrizarion.  The reason
is as follows.  Suppose that we have symmetrized the cochains $D^{n-k}_{j_1\cdots j_k}$  down  to
$k=3$.  The cochain $D^{n-2}_{j_1j_2}({\rm new})$  is

\hfil$D^{n-2}_{j_1j_2}-\sum_{j_3}E^{n-3}_{j_1j_2j_3}\cdot (z_{j_3}).$\hfil

\noindent  Then the sum

\hfil$\displaystyle \sum_ia_ic_{ij_1}P_{n-1}(x_i)-\sum_{j_2}D^{n-2}_{j_1j_2}({\rm new})\cdot (z_{j_2})$\hfil

\noindent  is the same as the old one since

$$\begin{aligned}
&  \sum_{j_2}\bigl(\sum_{j_3}E^{n-3}_{j_1j_2j_3}\cdot (z_{j_3})\bigr)(z_{j_2})\\
=&\sum_{j_2<j_3}(E^{n-3}_{j_1j_2j_3}-E^{n-3}_{j_1j_3j_2})(z_{j_3})\cdot (z_{j_2})\\
=&0
\end{aligned}
$$
since we have $E^{n-3}_{j_1j_2j_3}=E^{n-3}_{j_1j_3j_2}$.  This means that we do not need to change
the cochains $D^{n-1}_{j_1}$.  Since $\cP_n(X)$ is a cocycle  we have

\[\begin{aligned}
&d\cP_n(X)\\
=&\sum_{j_1}\sum_{j_2}D^{n-2}_{j_1j_2}({\rm new})\cdot (z_{j_2})\cdot (z_{j_1})\\
=&\sum_{j_1<j_2}(D^{n-2}_{j_1j_2}({\rm new})-\sum_{j_2}D^{n-2}_{j_2j_1}({\rm new}))
\cdot (z_{j_2})\cdot (z_{j_1})\\
=&0.
\end{aligned}
\]
So we have  $D^{n-2}_{j_1j_2}({\rm new})=D^{n-2}_{j_2j_1}({\rm new})$.  

\noindent We consider the first part of the assertion  (II) of (1).  We assume that 
the cochains $D^{n-k}_{j_1\cdots j_k}$ are symmetrized. Suppose that there is a $\wj\in J$
such that $c_{i\wj}=0$ for all $i$.  Let
\[\begin{aligned}
&\wk:={\rm max}\{k\in \ZZ\,|\,\text{ there exists a $k$-tuple } \\
&(j_1,\cdots, j_k) \text{ such that } \wj\in\{j_1,\cdots,  j_k\}  \text{ and } D^{n-k}_{j_1\cdots j_k}\neq 0.\}
\end{aligned}
\]
Let $K$ be the set of $\wk$-tuples $(j_1,\cdots j_{\wk})$ such that $\wj\in\{j_1,\cdots,  j_{\wk}\}$
and $D^{n-\wk}_{j_1\cdots j_{\wk}}\neq 0$.  The symmetric group on $\wk$  letters  act 
on $K$ by permutation.  Let $\co$  be an orbit of of the action.  For each $\wk$-tuple
$(j_1,\cdots j_{\wk})$  the cochain $D^{n-\wk}_{j_1\cdots j_{\wk}}$ is a cocycle,  so by Conjecture
\ref{BS0}  there is a  cochain $E^{n-\wk}_{j_1\cdots j_{\wk}}\in N^{-1}(n-\wk)$  such that
$dE^{n-\wk}_{j_1\cdots j_{\wk}}=D^{n-\wk}_{j_1\cdots j_{\wk}}$.  We assume that the cochains
$dE^{n-\wk}_{j_1\cdots j_{\wk}}$ are symmetric i.e. are independent of the ordering 
of $j_1,\cdots,  j_{\wk}.$  For eack $\wk-1$ tuple $(j_1,\cdots j_{\wk-1})$  let 
$D^{n-\wk+1}_{j_1\cdots j_{\wk-1}}(\text{new})$ be the cochain
\[D^{n-\wk+1}_{j_1\cdots j_{\wk-1}}+\sum_{j_{\wk},  \,\,(j_1,\cdots, j_{\wk-1},j_{\wk} )
\in \co}  E^{n-\wk}_{j_1\cdots j_{\wk}}\cdot (z_{j_{\wk}}).\]
When $\wk\geq 2$ for each $\wk-2$-tuple $(j_1,\cdots, j_{\wk-2})$  the cochain
\[\sum_ia_ic_{ij_1\cdots j_{\wk-2}}P_{n-\wk+2}(x_i)-\sum_{j_{\wk-1}}D^{n-\wk+1}_{j_1\cdots j_{\wk-1}}(\text{new})
\cdot (z_{j_{\wk-1}})\]
is the same as the old one since
\[\begin{aligned}
&\sum_{j_{\wk-1}}\biggl(\sum_{j_{\wk},\,\, (j_1,\cdots, j_{\wk-1},j_{\wk} )
\in \co}  E_{j_1\cdots j_{\wk}}\cdot (z_{j_{\wk}})\biggr)(z_{j_{\wk-1}})\\
=&\sum_{j_{\wk-1}<j_{\wk},\,\,(j_1,\cdots, j_{\wk-1},j_{\wk} )
\in \co}\bigl(E_{j_1\cdots j_{\wk-1} j_{\wk}}-E_{j_1\cdots j_{\wk} j_{\wk-1}}\bigr)
 (z_{j_{\wk}})\cdot (z_{j_{\wk-1}})\\
=&0.
\end{aligned}
\]
We perform this operation for each orbit of the group action,  and then symmetrize the cohcains 
$D^{n-k}_{j_1\cdots j_k}.$  By this we can reduce the number $\wk$ by 1. By repeating
this process $\wk$ is reduced to 1.  Then the cochain $D^{n-1}_{\wj}$ is a cocycle,  and so there
is a cochain  
\newline  $E^{n-1}_{\wj}\in N^{-1}(n-1)$ such that $dE^{n-1}_{\wj}=D^{n-1}_{\wj}$.
Then we change the cochain $\cP_n(X)$ to $\cP_n(X)+d(E^{n-1}_{\wj}\cdot (z_{\wj})).$
By performing this operation for each $\wj$ as above we obtain a cocycle
$\cP_n(X)$ which satisfy the first assertion of (II) of (1).

 \end{proof}

\section{The co-Lie algebra of Bloch-Kriz mixed Tate motives}
We recall the definition of the category of mixed Tate motives constructed by Bloch and Kriz
in \cite{BK}.  First  we recall  the definition  
of the bar complex $B(\cN)$ (see \cite{EM}, \cite{Hain} for details.) 
Let $\cN_+= \oplus_{r>0} N^\bullet(r)$.  As a vector space $B(\cN)$\index{$B(L, N,M)$} 
is equal to 
$\bigoplus_{s\ge 0} (\otimes ^s \cN_+)$, 
with the convention $(\otimes ^s \cN_+)=\QQ$ for $s=0$. 
An element $a_1\otimes \cdots \otimes a_s$  of 
 $\otimes^s   \cN_+$ 
  is written as
$[a_1|\cdots |a_s] $.

 The internal differential $d_I$ is defined by

\[ d_I([a_1|\cdots |a_s])
=\sum_{i=1}^{s}
(-1)^i [Ja_1|\cdots |Ja_{i-1}|da_{i}|\cdots |a_s] 
\]
where
$Ja=(-1)^{\deg a}a$.
The external differential $d_E$ is defined by
\[
 d_E( [a_1|\cdots |a_s] )
=\sum_{i=1}^{s-1}
(-1)^{i+1}[Ja_1|\cdots |(Ja_{i})a_{i+1}|\cdots |a_s].
\]
Then we have $d_Id_E+d_Ed_I=0$ and the map $d_E+d_I$ defines
a differential on $B(\cN)$. The degree of an element
$[a_1|\cdots |a_s]$
is defined to be $\sum_{i=1}^s \deg a_i -s$. There is a graded commutative 
shuffle product on $B(\cN)$  (\cite{May},  pp.  132-134)

\[[a_1|\cdots |a_r]\otimes [a_{r+1}|\cdots |a_{r+s}]\mapsto \sum_\mu(-1)^{\sigma(\mu)}
[a_{\mu(1)}|\cdots |a_{\mu(r+s)}]\]
where the sum is taken over the set of $(r,s)$ shuffles inside the  symmetric group
on $r+s$ letters, and the sign $(-1)^{\sigma(\mu)}$  is the sign of the graded permutation,
and where we take into account the degrees of $a_i$. When $\deg a_i=1$ for all $i$,
then $(-1)^{\sigma(\mu)}=1$ for all $\mu.$  The complex $B(\cN)$  also has a coproduct
\[\Delta:\,\,  B(\cN)\to B(\cN)\otimes B(\cN)\]
which is a map of complexes.  By this product and the coproduct 
the cohomology $H^0(B(\cN))$,  which we denote by
$\chi_F$  is a commutative Hopf algebra.  The category of mixed Tate motives which we denote
by  $\text{MT}_{\text BK}(F)$,   is defined by
Bloch and Kriz to be the category of Adams graded co-modules over $\chi_F$. $\chi_F$
is a direct sum  $\displaystyle \bigoplus_{r\geq 0} \chi_F(r)$,  and its augmentation
ideal is  $\displaystyle \bigoplus_{r> 0} \chi_F(r)$,  which we denote by $\chi_F^+$.  
The co-Lie algebra of $\text{MT}_{\text BK}(F)$  is by definition equal to $\chi_F^+/(\chi_F^+)^2$.

\noindent We recall the definition of motivic polylogarithms from \cite{BK}. 
For an element $x\in F^\times$  and $n\geq 2$,  the cochain
\[[\rho_n(x)]-[\rho_{n-1}(x)|(x)]+[\rho_{n-2}(x)|(x)|(x)]-\cdots +
(-1)^{(n-1)}[\rho_1(x)|(x)|\cdots |(x)]\]
of $B(\cN)$  is a cocycle of degree zero.  The class of this cocycle  in $H^0(B(\cN))=\chi_F$
is by definition the motivic polylogarithm $\bL_n(x)$.  $\bL_n(x)$  is contained in
the augmentation ideal $\chi_F^+$.
\begin{theorem}
\label{poly}
\begin{itemize}
\item[(1)]  Let $X=\sum_ia_i \{x_i\}$ be an element of $\Q[F^\times].$ Suppose that $X$ is in 
	$\Ker \delta_n'$.  Then the cocycle $\cP_n(X)$  defined in the assertion {\rm (III)} of 
	{\rm (1)}
	in {\rm Theorem \ref{existc}}  defines the same element  as $\sum_i a_i\bL_n(x_i)$
	of $\chi_F^+/(\chi_F^+)^2$.
	
	\item[(2)]  If  the element $X$ is in $R_n'(F)$,  then the element   $\sum_i a_i\bL_n(x_i)$
	in $\chi_F^+/(\chi_F^+)^2$ is zero.
\end{itemize}
\end{theorem}
\begin{proof} (1).  In the following  $A\equiv B$  means that $A-B$  is contained in 
coboundary of the bar complex $B(\cN).$  We have
\[d[D_{j_1}^{n-1}|(z_{j_1})]=-[dD_{j_1}^{n-1}|(z_{j_1})]+[D_{j_1}^{n-1}\cdot (z_{j_1})]\]
and 
\[dD^{n-1}_{j_1}=\sum_ia_ic_{ij_1}P_{n-1}(x_i)-\sum_{j_2}D^{n-2}_{j_1j_2}z_{j_2}\]
so that
\[\begin{aligned}
&\sum_{j_1}[D^{n-1}_{j_1}\cdot (z_{j_1})]\\
\equiv& \sum_{j_1}[dD^{n-1}_{j_1}|(z_{j_1})]\\
=&\bigl[\sum_ia_iP_{n-1}(x_i)|x_i[z]\bigr]-
\bigl[\sum_{j_1,j_2}D^{n-2}_{j_1j_2}\cdot (z_{j_2})|(z_{j_1})\bigr].
\end{aligned}
\]
We have
\[\begin{aligned}
&\sum_{j_1,j_2}d[D_{j_1j_2}^{n-2}|(z_{j_2})|(z_{j_1})]\\
=&\sum_{j_1,j_2}\bigl(-[dD^{n-2}_{j_1j_2}|(z_{j_2})|(z_{j_1})]+[D^{n-2}_{j_1j_2}\cdot (z_{j_2})|(z_{j_1})]
+[D^{n-2}_{j_1j_2}|(z_{j_2})\cdot (z_{j_1})]\bigr)
\end{aligned}
\]
and the sum 
\[\sum_{j_1,j_2}[D^{n-2}_{j_1j_2}|(z_{j_2})\cdot (z_{j_1})]=0\]
since we have 
\[D^{n-2}_{j_1j_2}=D^{n-2}_{j_2j_1}.\]
Hence we have
\[\begin{aligned}
&\bigl[\sum_{j_1,j_2}D^{n-2}_{j_1j_2}\cdot (z_{j_2})|(z_{j_1})\bigr]\\
\equiv&\sum_{j_1,j_2}[dD^{n-2}_{j_1j_2}|(z_{j_2})|(z_{j_1})]\\
=& \sum_{j_1,j_2}[\sum_ia_ic_{ij_1j_2}P_{n-2}(x_i)-\sum_{j_3}D^{n-3}_{j_1j_2j_3}\cdot (z_{j_3})|(z_{j_2})|(z_{j_1})]\\
=& \sum_i[a_iP_{n-2}(x_i)|x_i[z]|x_i[z]]-\sum_{j_1,j_2,j_3}[D^{n-3}_{j_1j_2j_3}\cdot (z_{j_3})|(z_{j_2})|(z_{j_1})].
\end{aligned}
\]
Continuing this computation we see that
\[\begin{aligned}
&[\cP_n(X)]\\
\equiv& \sum_{k=0}^{n-3}\sum_i(-1)^ka_i[P_{n-k}(x_i)|x_i[z]^{\otimes k}]+(-1)^{n-2}
\sum_{j_1,\cdots, j_{n-2}} [dD^2_{j_1\cdots j_{n-2}}|(z_{j_{n-2}})|(z_{j_{n-3}})|\cdots |(z_{j_1})].
\end{aligned}
\]
We compute $[P_n(x_i)].$ For $1\leq  k\leq n-1$ we have the equality
\[\begin{aligned}
&d[\rho_{n-k}(x)|r(x,z)^k]\\
=&-[\rho_{n-k-1}(x)\cdot (x)|r(x,z)^k]-k[\rho_{n-k}(x)|r(x,z)^{k-1}\cdot \bigl((x)-x[z]\bigr)]\\
&-[\rho_{n-k}(x)\cdot r(x,z)^k]
\end{aligned}
\]
where we set $\rho_0(x)=0.$  We have 
\[\begin{aligned}
& [\rho_{n-k}(x)\cdot r(x,z)^k]\\
\equiv &-[\rho_{n-k-1}(x)\cdot (x)|r(x,z)^k]-k[\rho_{n-k}(x)|r(x,z)^{k-1}\cdot \bigl((x)-x[z]\bigr)].
\end{aligned}
\]
We have
\[\begin{aligned}
&d[\rho_{n-k-1}(x)|(x)|r(x,z)^k]\\
=& -[\rho_{n-k-2}(x)\cdot (x)|(x)|r(x,z)^k]-k[\rho_{n-k-1}(x)|(x)|r(x,z)^{k-1}\cdot \bigl((x)-x[z]\bigr)]\\
&-[\rho_{n-k-1}(x)\cdot (x)|r(x,z)^k]-[\rho_{n-k-1}(x)|(x)\cdot r(x,z)^k]
\end{aligned}
\]
and so that
\[\begin{aligned}
&-[\rho_{n-k-1}(x)\cdot (x)|r(x,z)^k]\\
\equiv&[\rho_{n-k-2}(x)\cdot (x)|(x)|r(x,z)^k]+k[\rho_{n-k-1}(x)|(x)|r(x,z)^{k-1}\cdot \bigl((x)-x[z]
\bigr)] \\
&+[\rho_{n-k-1}(x)|(x)\cdot r(x,z)^k].
\end{aligned}
\]
Continuing this computation we obtain the following.
\[
\begin{aligned}
&[\rho_{n-k}(x)\cdot r(x,z)^k]\\
\equiv& k\biggl(-[\rho_{n-k}(x)|+[\rho_{n-k-1}(x)|(x)|-\cdots +(-1)^{n-k}[\rho_1(x)|
(x)^{\otimes (n-k-1)}|\biggr)\\
&\otimes r(x,z)^{k-1}\cdot\bigl((x)-x[z]\bigr)]\\
&+\biggl([\rho_{n-k-1}(x)|-[\rho_{n-k-2}(x)|(x)|+\cdots +(-1)^{n-k-2}[\rho_1(x)|(x)^{\otimes (n-k-2)}|\biggr)
(x)\cdot r(x,z)^k].
\end{aligned}
\]
Multiplying this with $\frac1{k!}$ we obtain

\[
\begin{aligned}
&\frac1{k!}[\rho_{n-k}(x)\cdot r(x,z)^k]\\
\equiv& \frac1{(k-1)!}\biggl(-[\rho_{n-k}(x)|+[\rho_{n-k-1}(x)|(x)|-\cdots 
+(-1)^{n-k}[\rho_1(x)|(x)^{\otimes (n-k-1)}|\biggr)\\
&\otimes r(x,z)^{k-1}\cdot \bigl((x)-x[z]\bigr)]\\
&+\frac1{k!}\biggl([\rho_{n-k-1}(x)|-[\rho_{n-k-2}(x)|(x)|+\cdots 
+(-1)^{n-k-2}[\rho_1(x)|(x)^{\otimes (n-k-2)}|\biggr)(x)\cdot r(x,z)^k].
\end{aligned}
\]

By taking the sum we obtain
\[
\begin{aligned}
&\sum_{k=1}^{n-1}\frac1{k!}[\rho_{n-k}(x)\cdot r(x,z)^k]\\
\equiv& \biggl(-[\rho_{n-1}(x)|+[\rho_{n-2}(x)|(x)|-\cdots 
+(-1)^{n-1}[\rho_1(x)|(x)^{\otimes (n-2)}|\biggr)\bigl((x)-x[z])\bigr]\\
&+\sum_{k=2}^{n-1}\frac1{(k-1)!}\biggl([\rho_{n-k}(x)|-[\rho_{n-k-1}(x)|(x)|-\cdots +(-1)^{n-k-1}
[\rho_1(x)|(x)^{\otimes (n-k-1)}|\biggr)\\
&\otimes r(x,z)^{k-1}\cdot x[z]]\\
=&-[\rho_{n-1}(x)|(x)]+[\rho_{n-2}(x)|(x)|(x)]-\cdots +(-1)^{n-1}[\rho_1(x)|(x)^{\otimes (n-1)}]\\
&+\sum_{k=1}^{n-1}\frac1{(k-1)!}\biggl([\rho_{n-k}(x)|-[\rho_{n-k-1}(x)|(x)|-\cdots +(-1)^{n-k-1}
[\rho_1(x)|(x)^{\otimes (n-k-1)}|\biggr)\\
&\otimes r(x,z)^{k-1}\cdot x[z]]
\end{aligned}
\]
Hence we have
\[\begin{aligned}
&[P_n(x)]\\
=&[\rho_n(x)]+\sum_{k=1}^{n-1}\frac1{k!}[\rho_{n-k}(x)\cdot r(x,z)^k]\\
\equiv& [\rho_n(x)]-[\rho_{n-1}(x)|(x)]+[\rho_{n-2}(x)|(x)|(x)]-\cdots
 +(-1)^{n-1}[\rho_1(x)|(x)^{\otimes (n-1)}]\\
&+\sum_{k=1}^{n-1}\frac1{(k-1)!}\biggl([\rho_{n-k}(x)|-[\rho_{n-k-1}(x)|(x)|-\cdots +(-1)^{n-k-1}
[\rho_1(x)|(x)^{\otimes (n-k-1)}|\biggr)\\
&\otimes r(x,z)^{k-1}\cdot x[z]]\\
=&\bL_n(x)\\
+&\sum_{k=1}^{n-1}\frac1{(k-1)!}\biggl([\rho_{n-k}(x)|-[\rho_{n-k-1}(x)|(x)|-\cdots +(-1)^{n-k-1}
[\rho_1(x)|(x)^{\otimes (n-k-1)}|\biggr)\\
&\otimes r(x,z)^{k-1}\cdot x[z]].
\end{aligned}
\]

We compute $[P_{n-1}(x_i)|x_i[z]].$ For $1\leq k\leq n-2$ we have
\[\begin{aligned}
&d\bl\rho_{n-1-k}(x)|r(x,z)^k|x[z]\br\\
=&-\bl\rho_{n-1-k-1}(x)\cdot(x)|r(x,z)^k|x[z]\br-k\bl\rho_{n-1-k}(x)|r(x,z)^{k-1}
\bigl((x)-x[z]\bigr)|x[z]\br\\
&-\bl\rho_{n-1-k}(x)\cdot r(x,z)^k|x[z]\br+\bl\rho_{n-1-k}(x)|r^k(x,z)\cdot x[z]\br.
\end{aligned}
\]
so that we have
\[\begin{aligned}
&\bl\rho_{n-1-k}(x)\cdot r(x,z)^k|x[z]\br\\
\equiv&-\bl\rho_{n-1-k-1}(x)\cdot (x)|r(x,z)^k|x[z]\br-k\bl\rho_{n-1-k}(x)|r(x,z)^{k-1}
\bigl((x)-x[z]\bigr)|x[z]\br\\
&+\bl\rho_{n-1-k}(x)|r(x,z)^k\cdot x[z]\br.
\end{aligned}
\]
We also have
\[\begin{aligned}
&d\bl\rho_{n-1-k-1}(x)|(x)|r(x,z)^k|x[z]\br\\
=& -\bl\rho_{n-1-k-2}(x)\cdot (x)|(x)|r(x,z)^k|x[z]\br-k\bl\rho_{n-1-k-1}(x)|(x)|
r(x,z)^{k-1}\cdot \bigl((x)-x[z]\bigr)|x[z]\br\\
&-\bl \rho_{n-1-k-1}(x)\cdot (x)|r(x,z)^k|x[z]\br-\bl \rho_{n-1-k-1}(x)|(x)\cdot r(x,z)^k|x[z]\br\\
&+[\rho_{n-1-k-1}(x)|x|r(x,z)^k\cdot x[z]\br
\end{aligned}
\]
and so that
\[\begin{aligned}
&-\bl \rho_{n-1-k-1}(x)\cdot(x)|r(x,z)^k|x[z]\br\\
\equiv&\bl\rho_{n-1-k-2}(x)\cdot (x)|(x)|r(x,z)^k|x[z]\br+k\bl\rho_{n-1-k-1}(x)|(x)|
r(x,z)^{k-1}\cdot \bigl((x)-x[z]\bigr)|x[z]\br\\
&+\bl \rho_{n-1-k-1}(x)|(x)\cdot r(x,z)^k|x[z]\br-\bl\rho_{n-1-k-1}(x)|(x)|r(x,z)^k\cdot x[z]\br.
\end{aligned}
\]
Continuing this computation we obtain the following.
\[\begin{aligned}
&\frac1{k!}\bl\rho_{n-1-k}(x)\cdot r(x,z)^k|x[z]\br\\
\equiv&\frac1{(k-1)!}\biggl(-\bl\rho_{n-1-k}(x)|+\bl\rho_{n-1-k-1}(x)|(x)|-\cdots +(-1)^{n-1-k}
\bl\rho_1(x)|(x)^{\otimes (n-1-k-1)}|\biggr)\\
&\otimes r(x,z)^{k-1}\cdot\bigl((x)-x[z]\bigr)|x[z]\br\\
+&\frac1{k!}\biggl(\bl\rho_{n-1-k}(x)|-\bl\rho_{n-1-k-1}(x)|x|+\cdots +(-1)^{n-1-k-1}
\bl\rho_1(x)|(x)^{\otimes (n-1-k-1)}|\biggr)\\
&\otimes r(x,z)^k\cdot x[z]\br\\
&+\frac1{k!}\biggl(\bl \rho_{n-1-k-1}(x)|-\bl \rho_{n-1-k-2}(x)|x|+\cdots +(-1)^{n-k-3}
\bl \rho_1(x)|(x)^{\otimes (n-1-k-2)}|\biggr)\\
&\otimes (x)\cdot r(x,z)^k|x[z]\br.
\end{aligned}
\]
By taking the sum we obtain the following.
\[\begin{aligned}
&[P_{n-1}(x)|x[z]]\\
=&\sum_{k=0}^{n-2}\frac1{k!}\bl\rho_{n-1-k}(x)\cdot r(x,z)^k|x[z]\br\\
=& [\rho_{n-1}(x)|x[z]] +\sum_{k=1}^{n-2}\frac1{k!}\bl\rho_{n-1-k}(x)\cdot r(x,z)^k|x[z]\br\\
\equiv&\sum_{k=0}^{n-2}\frac1{k!}\bigl(\bl\rho_{n-1-k}(x)|-\bl\rho_{n-1-k-1}(x)|x|+\cdots 
+(-1)^{n-1-k-1}\bl\rho_1(x)|(x)^{\otimes (n-1-k-1)}|\bigr)r(x,z)^k\cdot x[z]\br\\
&+\sum_{k=1}^{n-2}\frac1{(k-1)!}\bigl(\bl \rho_{n-1-k}(x)|-\bl\rho_{n-1-k-1}(x)|x|+\cdots +
(-1)^{n-1-k-1}\bl\rho_1(x)|(x)^{\otimes (n-1-k-1)}\bigr)\\
&\otimes r(x,z)^{k-1}\cdot x[z]|x[z]\br.
\end{aligned}
\]
By similar computation, for $1\leq t\leq n-3$ we have
\[\begin{aligned}
&[P_{n-t}(x)|x[z]^{\otimes t}]\\
=&\sum_{k=0}^{n-t-1}\frac1{k!}\bl\rho_{n-t-k}(x)\cdot r(x,z)^k|x[z]^{\otimes t}\br\\
=& [\rho_{n-t}(x)|x[z]^{\otimes t}] +\sum_{k=1}^{n-t-1}\frac1{k!}\bl\rho_{n-1-k}(x)
\cdot r(x,z)^k|x[z]^{\otimes t}\br\\
\equiv&\sum_{k=0}^{n-t-1}\frac1{k!}\biggl(\bl\rho_{n-t-k}(x)|-\bl\rho_{n-t-k-1}(x)|(x)|+\cdots 
+(-1)^{n-t-k-1}\bl\rho_1(x)|(x)^{\otimes (n-t-k-1)}\biggr)\\
&\otimes r(x,z)^k\cdot x[z]|x_i[z]^{\otimes (t-1)}\br\\
&+\sum_{k=1}^{n-t-1}\frac1{(k-1)!}\biggl(\bl \rho_{n-t-k}(x)|-\bl\rho_{n-t-k-1}(x)|(x)|+\cdots +
(-1)^{n-t-k-1}\bl\rho_1(x)|(x)^{\otimes (n-t-k-1)}\biggr)\\
&\otimes r(x,z)^{k-1}\cdot x[z]|x[z]^{\otimes t}\br.
\end{aligned}
\]
By taking the sum we obtain
\[\begin{aligned}
&\sum_{t=0}^{n-3}\sum_i(-1)^ta_i[P_{n-t}(x_i)|x_i[z]^{\otimes t}]\\
=&\sum_ia_i \bL_n(x_i)\\
+&(-1)^{n-3}a_i \biggl(\Bigl([\rho_2(x_i)|-[\rho_1(x_i)|(x_i)|\Bigr)x_i[z]|x_i[z]^{\otimes (n-3)}]\\
+&[\rho_1(x_i)|r(x_i,z)\cdot x_i[z]|x_i[z]^{\otimes (n-3)}]\biggr).
\end{aligned}
\]
By the assertion (I) of Theorem \ref{existc} we have
\[
\begin{aligned}
&\sum_{j_1,\cdots, j_{n-2}} [dD^2_{j_1\cdots j_k}|(z_{j_{n-2}})|(z_{j_{n-3}})|\cdots |(z_{j_1})]\\
=&\sum_ia_i[\widetilde{P}_2(x_i)|x_i[z]^{\otimes (n-2)}]\\
=&\sum_ia_i\Bigl([\rho_2(x_i)+\rho_1(x_i)r(x_i,z)-r(\rho_1(x_i),z)x_i[z]|\Bigr)x_i[z]^{\otimes (n-2)}].
\end{aligned}
\]
By computation we have
\[\begin{aligned}
&[\rho_1(x)\cdot r(x,z)|x[z]^{\otimes (n-2)}]\\
\equiv&-[\rho_1(x)|\bigr((x)-x[z]\bigr)|x_i[z]^{\otimes (n-2)}]+[\rho_1(x)|r(x,z)\cdot 
x[z]|x[z]^{\otimes (n-3)}]
\end{aligned}
\]
and 
\[
-[r(\rho_1(x),z)x[z]|x_i[z]^{\otimes (n-2)}]
\equiv-[\rho_1(x)-\rho_1(x)[z]|x[z]^{\otimes (n-1)}].
\]
Hence we have
\[\begin{aligned}
&[\rho_2(x)+\rho_1(x)r(x,z)-r(\rho_1(x),z)x[z]|x[z]^{\otimes (n-2)}]\\
\equiv&\bigl([\rho_2(x)|-[\rho_1(x)|(x)|\bigr)x[z]^{\otimes (n-2)}]
+[\rho_1(x)|r(x,z)x[z]|x_i[z]^{\otimes (n-3)}]\\
&+[\rho_1(x)[z]|x_i[z]^{\otimes (n-1)}].
\end{aligned}
\]
By taking sum we obtain
\[\begin{aligned}
&[\cP_n(X)]\\
\equiv& \sum_{k=0}^{n-3}\sum_i(-1)^ka_i[P_{n-k}(x_i)|x_i[z]^{\otimes k}]+(-1)^{n-2}
\sum_{j_1,\cdots, j_{n-2}} [dD^2_{j_1\cdots j_k}|z_{j_{n-2}}|z_{j_{n-3}}|\cdots |z_{j_1}]\\
=&\sum_ia_i\bigl(\bL_n(x_i)+(-1)^{n-2}[\rho_1(x_i)[z]|x_i[z]^{\otimes (n-1)}]\bigr).
\end{aligned}
\]
For each $(n-2)$-tuple $j_1,\cdots, j_{n-2}$ we have
\begin{equation}
\label{sym}
\begin{aligned}
&d^2D^2_{j_1\cdots j_{n-2}}=\sum_ia_ic_{ij_1\cdots j_{n-2}}\rho_1(x_i)[z]\cdot x_i[z]\\
=&0.
\end{aligned}
\end{equation}
We also have
\[\begin{aligned}
&\sum_ia_i[\rho_1(x_i)[z]|x_i[z]^{\otimes (n-1)}]\\
=&\sum_{j_1,\cdots, j_{n-2}}\sum_i a_ic_{ij_1\cdots j_{n-2}}
[\rho_1(x_i)[z]|x_i[z]|z_{j_{n-2}}|z_{j_{n-3}}|\cdots |z_{j_1}].
\end{aligned}
\]
By (\ref{sym}) it follows that the coefficient of 
$[z_{j_n}|z_{j_{n-1}}|z_{j_{n-2}}|z_{j_{n-3}}|\cdots |z_{j_1}]$ in the sum

\hfil$\sum_i a_ic_{ij_1\cdots j_{n-2}}[\rho_1(x_i)[z]|x_i[z]|z_{j_{n-2}}|z_{j_{n-3}}|\cdots |z_{j_1}]$\hfil

\noindent is the same as the coefficient of 
$[z_{j_{n-1}}|z_{j_n}|z_{j_{n-2}}|z_{j_{n-3}}|\cdots |z_{j_1}]$.  Hence the coefficients
of $[z_{j_n}|z_{j_{n-1}}|z_{j_{n-2}}|z_{j_{n-3}}|\cdots |z_{j_1}]$  are independent of the
ordering of $j_1,\cdots, j_n$.  It follows that the sum

\[\sum_ia_i [\rho_1(x_i)[z]|x_i[z]^{\otimes (n-1)}]\]
is a linear combination of the shuffle products

\hfil$[z_{j_1}]\cdot[z_{j_2}] \cdots  [z_{j_n}].$\hfil

\vskip 0.5cm

(2).  If $X$ is in $R_n'(F)$  then by (2) of Theorem \ref{existc}  $\cP_n(X)$ is a coboundary,
which means the class of $\cP_n(X)$ in $\chi_F$ is zero.  The assertion follows from (1) and this.
\end{proof}

\vskip 0.5cm

Competing interests: The author declares none.

\vskip 0.5cm

Kenichiro Kimura

Department of Mathematics,

University of Tsukuba,  

1-1-1 Tennnodai,  Tsukuba, 305-0006,  

Japan. 

email:   kimurak@math.tsukuba.ac.jp

\end{document}